%% file: main.tex
\documentclass{article}

\usepackage[preprint]{neurips_2026}

\usepackage[utf8]{inputenc} 
\usepackage[T1]{fontenc}    
\usepackage{hyperref}       
\usepackage{url}            
\usepackage{booktabs}       
\usepackage{amsfonts}       
\usepackage{nicefrac}       
\usepackage{microtype}      
\usepackage{xcolor}         

\input{macros.tex}

\title{No-regret optimization of time-varying bilevel problems}

\author{%
  Eliabelle Mauduit \thanks{Use footnote for providing further information
    about author (webpage, alternative address)---\emph{not} for acknowledging
    funding agencies.} \\
  Unité de Mathématiques Appliquées \\
  ENSTA \\
  Institut Polytechnique de Paris\\
  Palaiseau, 91120\\
  \texttt{eliabelle.mauduit@ensta.fr} \\
  \And
  Eloïse Berthier \\
  Unité d'Informatique et d'Ingénierie des Systèmes\\
  ENSTA\\
  Institut Polytechnique de Paris \\
  Palaiseau, 91120 \\
  \texttt{eloise.berthier@ensta.fr} \\
  \AND
  Andrea Simonetto \\
  Unité de Mathématiques Appliquées\\
  ENSTA \\
  Institut Polytechnique de Paris \\
  Palaiseau, 91120 \\
  \texttt{andrea.simonetto@ensta.fr} \\
}

\begin{document}

\maketitle

\begin{abstract}
Bilevel optimization problems arise in many applications where decisions must account for the optimal response of another system, such as in game-theoretic settings. However, these problems are notoriously challenging, as even linear bilevel programs are strongly NP-hard.
In this work, we consider bilevel optimization with a known upper-level objective and an unknown, potentially time-varying lower-level response, accessible only through noisy zeroth-order observations. We propose W-SparQ-BL, a Bayesian optimization framework that models the lower-level mapping using multi-output Gaussian processes and enables efficient optimization under uncertainty.
Our approach leverages a sparse, observation-based approximation to control the effect of noise and temporal variability, while requiring only limited access to additional information over time. We establish regularity results linking the lower-level response to standard RKHS assumptions for common kernels, including Matérn and squared exponential.
We prove that W-SparQ-BL achieves sublinear dynamic regret in both stationary and time-varying settings. Experiments on  representative time-varying game-theoretic problems demonstrate the effectiveness of our approach.
\end{abstract}

\section{Introduction}

In many modern applications, evaluating an objective function is costly, noisy, and must be done sequentially. Additionally, the objective function might depend on another unknown system's response. For instance, the electricity consumption of clients responds to energy prices proposed by a provider, and this response cannot be fully modeled by the provider optimizing its revenue.

This setting naturally calls for adaptive sampling strategies that balance exploration and exploitation. When structural assumptions on the objective can be encoded through a kernel, Gaussian Process (GP) regression~\citep{Rasmussen2006Gaussian} provides a principled framework for nonparametric modeling via posterior mean and uncertainty quantification. Two complementary perspectives coexist in the literature. In the Bayesian setting, the objective function $f$ is modeled as a sample path of a GP, $f \sim \mathcal{GP}(0,k)$, where the kernel $k$ encodes prior assumptions. In the frequentist setting, $f$ is assumed to belong to a reproducing kernel Hilbert space (RKHS) with bounded norm, and GP regression is used as a surrogate model. These viewpoints are fundamentally distinct: GP sample paths are almost surely rougher than RKHS functions~\citep{wahba1990spline}, yet both lead to effective algorithms. In this work, we propose to extend GP-based optimization strategies to time-varying bilevel problems.

\paragraph{Bayesian optimization and GP-UCB.}

Bayesian optimization leverages GP models to optimize black-box functions under limited evaluation budgets~\citep{Rasmussen2006Gaussian}. Given observations $f(x_1),\dots,f(x_T)$, the GP posterior provides a predictive distribution for any new input, enabling uncertainty-aware decision rules. A prominent example is \textbf{GP-UCB}~\citep{srinivas2012information}, which selects actions by maximizing an upper confidence bound and achieves sublinear regret when the objective is time-invariant.

\paragraph{Time variation and its limitations.}

In many real-world scenarios, however, the objective evolves over time. In such settings, classical bandit feedback is no longer sufficient to guarantee no-regret. Existing analyses typically introduce a variation budget
\begin{equation}\label{Eq:GeneralVarBudget}
V_T = \sum_{t=1}^{T-1} \|f_{t+1} - f_t\|_{\mathcal{H}_k},
\end{equation}
and regret bounds degrade with $V_T$. In fact, linear regret is unavoidable under certain variations regimes when only bandit feedback is available~\citep{bogunovic2016,imamura2020timevarying,deng2022weighted,DBLP:Zhoujournals/corr/abs-2102-06296}. These limitations motivate relaxing the strict bandit setting by allowing additional observations. Under H\"older-type temporal assumptions, \textbf{SparQ-GP-UCB}~\citep{mauduit2025time} and its windowed variant \textbf{W-SparQ-GP-UCB}~\citep{mauduit2025no} leverage sparse queries on previously observed inputs to recover sublinear dynamic regret guarantees in non-stationary environments.

\paragraph{Sequential decision-making with time-varying rewards.}

We consider the following general framework. At each time $t$, the learner selects an action $x_t \in \mathcal{D}$ and observes
\begin{equation}\label{Eq:BaseModel}
y_t = f(x_t,t) + \epsilon_t, \qquad \epsilon_t \sim \mathcal{N}(0,\sigma_t^2),
\end{equation}
where the noise variance may increase with time, reflecting the progressive obsolescence of past observations. The goal is to minimize the dynamic regret
\begin{equation}\label{Eq:DynamicRegret}
R_T = \sum_{t=1}^T \left(\max_{x \in \mathcal{D}} f_t(x) - f_t(x_t)\right).
\end{equation}
In this setting, GP models enable uncertainty-aware exploration, while additional feedback mechanisms compensate for temporal variations.

\paragraph{Sequential games.}

Beyond single-agent optimization, many applications involve interactions with strategic or reactive agents. As an example, we can cite sequential games, where the learner’s reward depends on the response of a follower. A commonly used formulation is the following, where at each round,
\begin{equation}\label{Eq:SeqGames}
\begin{aligned}
x_{t+1} \in \argmax_{x \in \mathcal{D}} \quad & f(x,y_t) \\
\text{s.t.} \quad & y_t = g(x_t,\theta_t),
\end{aligned}
\end{equation}
where $\theta_t$ represents the follower’s strategy. This formulation is closely related to Stackelberg games~\citep{bruckner2011stackelberg}. In the setting of~\cite{sessa2020learning}, the response function $g$ is assumed to lie in an RKHS and is learned via GP regression, leading to sublinear static regret
\begin{equation}\label{Eq:SessaRegret}
\tilde{R}_T = \max_{x \in \mathcal{D}} \sum_{t=1}^T f(x,y_t) - \sum_{t=1}^T f(x_t,y_t).
\end{equation}
Our first contribution studies this specific problem where we consider instead the dynamic regret~\eqref{Eq:DynamicRegret}, which is more appropriate for time-varying environments~\citep{jacobsendynamic}, and we exploit additional feedback through sparse additional queries.

\paragraph{Bilevel Bayesian optimization.} A more natural and widely studied formulation assumes that the follower's response is defined implicitly as
\begin{align*}
y_t^*(x) \in \argmin_{y \in \mathcal{Y}} g_t(x,y),
\end{align*}
leading to a bilevel optimization problem. Recent works such as BILBAO~\citep{ekmekcioglu2024bayesian} and BILBO~\citep{chew2025bilbo} propose GP-based approaches for such problems. On the one hand, BILBAO jointly optimizes the upper and lower levels of the bilevel problem using a joint GP and although no regret analysis is conducted, the method demonstrates good empirical performances. On the other hand, BILBO builds trusted sets of solutions on which it is easier to bound the regret to derive sublinear bounds. However, the assumptions on the lower level function $g_t$ ensuring RKHS structure of the argmin response map are not explicitly characterized. Both approaches are defined for time-invariant problems.

\paragraph{Contributions.}

In this work, we provide a theoretical and algorithmic framework for time-varying bilevel optimization under limited feedback. Our main contributions are as follows:
\begin{itemize}
    \item In Section~\ref{Sec:SequentialGames}, we provide an algorithm, based on \textbf{W-SparQ-GP-UCB}, which achieves sublinear dynamic regret for time-varying sequential games.
    \item In Section~\ref{Sec:bilevel}, we characterize conditions under which the argmin response map in bilevel problems belongs to an RKHS, with a comparison between Matérn and squared exponential kernels.
    \item We then extend \textbf{W-SparQ-GP-UCB} to bilevel problems, yielding the algorithm \textbf{W-SparQ-BL}, for which we establish sublinear dynamic regret guarantees under controlled temporal variation and sparse additional feedback.
    \item Finally, in Section~\ref{Sec:numerics}, we numerically evaluate the algorithms on time-varying problems.
\end{itemize}
Table~\ref{tab:tableContributions} positions the contributions with respect to existing results. To the best of our knowledge, this is the first work providing no-regret guarantees for time-varying bilevel Bayesian optimization with unknown lower-level objectives.

\begin{table}
    \centering
    \caption{Summary of existing \textit{no-regret} optimization algorithms for different online optimization problems, with time-invariant and time-varying versions, along with the present contributions in bold.\\}
    \begin{tabular}{l | m{9em} | m{9em}| m{9em}  } 
         & Bayesian optimization & Sequential games & Bilevel optimization\\ \hline 
        Time-invariant & GP-UCB \citep{srinivas2012information} &  StackelUCB \citep{sessa2020learning} &  \textbf{GP-UCBL (Section \ref{sec:tibi})}\\ \hline 
        Time-varying & W-SparQ-GP-UCB \citep{mauduit2025no} & \textbf{Section~\ref{Sec:SequentialGames}} & \textbf{W-SparQ-BL (Section~\ref{sec:tvbi})}\\ \hline 
    \end{tabular}
    \label{tab:tableContributions}
\end{table}

\section{Bayesian optimization of time-varying sequential games}\label{Sec:SequentialGames}

We first consider sequential decision-making problems involving strategic interactions, where a learner repeatedly interacts with an opponent whose behavior evolves over time, formalized in Problem~\eqref{Eq:SeqGames}. Our goal is to design a learning strategy that adapts to these temporal variations and tracks the sequence of optimal decisions. This setting corresponds to the sequential games framework studied in~\citep{sessa2020learning}, to which we compare throughout this section.

\subsection{Problem formulation}

At each round $t \geq 1$, the learner selects an action $x_t \in \mathcal{X}$, where $\mathcal{X} \subset \mathbb{R}^d$ is a convex compact set, and faces an opponent whose response depends on a latent type $\theta_t \in \Theta$, where $\Theta \subset \mathbb{R}^p$ is a discrete set of opponents. The learner receives a known reward
$f : \mathcal{X} \times \mathcal{Y} \to \mathbb{R}$, 
evaluated at the pair $(x_t, y_t)$, where the opponent’s response is given by
\begin{align*}
y_t = g(x_t,\theta_t) + \epsilon_t, \qquad \epsilon_t \sim \text{subG}(\sigma_t^2) ,
\end{align*}
with $g : \mathcal{X} \times \Theta \to \mathcal{Y}$ unknown, $\mathcal{Y} \subset \mathbb{R}^m$  convex compact and $\epsilon_t$ is the noise term, wich is sub-Gaussian with variance proxy~$\sigma_t^2$~\citep{buldygin1980sub}. Importantly, the opponent’s type $\theta_t$ is revealed only \emph{after} the learner selects $x_t$, so decisions must be made without knowledge of the current opponent state. A typical interaction is illustrated in Figure~\ref{Fig:SequentialGames}.

\begin{figure}
	\centering
	\includegraphics[scale=0.6]{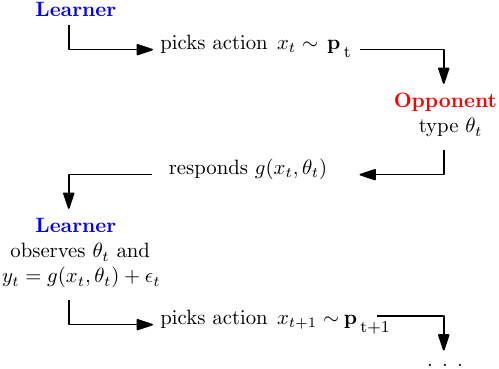}
	\caption{One step of a sequential game.}
	\label{Fig:SequentialGames}
\end{figure}

\paragraph{Baseline: static regret and adversarial opponents.}

In~\cite{sessa2020learning}, the sequence $(\theta_t)$ is allowed to be fully adversarial. As a consequence, the learner cannot anticipate future opponent behavior, and performance is measured using a \emph{static regret}
\begin{align*}
\tilde{R}_T = \max_{x \in \mathcal{X}} \sum_{t=1}^T f(x,y_t) - \sum_{t=1}^T f(x_t,y_t),
\end{align*}
which compares the learner to the best fixed action in hindsight. While this notion enables robustness to arbitrary opponent strategies, it ignores the fact that the optimal action typically depends on $\theta_t$ and may vary across rounds. In particular, it does not capture the learner’s ability to track a time-varying optimum. In contrast, we aim to compete against the sequence of optimal actions
\begin{align*}
x_t^* = \arg\max_{x \in \mathcal{X}} f(x, g(x,\theta_t)),
\end{align*}
and measure performance through the dynamic regret
\begin{align*}
R_T = \sum_{t=1}^T f(x_t^*, y_t) - \sum_{t=1}^T f(x_t, y_t).
\end{align*}

Achieving sublinear dynamic regret is impossible under fully adversarial opponent sequences~\citep{bogunovic2016}. To make this objective attainable, we introduce a stochastic model for the evolution of the opponent’s type.

\begin{assumption}[Opponent's type variations]\label{As:opponentVar}
Let $L_g$ denote the Lipschitz constant of $g$. There exists $\alpha \geq 0$ such that, for any $t_1 < t_2$,
\begin{align*}
\theta_{t_2} = \theta_{t_1} + \xi_{t_1,t_2},
\end{align*}
where $\xi_{t_1,t_2}$ is zero-mean sub-Gaussian with variance proxy
\begin{align*}
\frac{\sigma^2}{L_g^2} |t_2 - t_1|^\alpha,
\end{align*}
and, for fixed $t_2$, the variables $\{\xi_{t_1,t_2}\}_{t_1<t_2}$ are independent.
\end{assumption}

This assumption models a wide range of realistic scenarios where the opponent evolves smoothly over time (e.g., drifting preferences, adaptive agents, wear and tear, or physical systems with inertia). To make the time dependence explicit, we define $
g_t(x) = g(x,\theta_t)$.

\begin{proposition}\label{Prop:responseReg}
Under Assumption~\ref{As:opponentVar}, for all $t_1 < t_2$ and $x \in \mathcal{X}$,
\begin{align*}
g_{t_1}(x) = g_{t_2}(x) + v_{t_1,t_2},
\end{align*}
where $v_{t_1,t_2}$ is zero-mean sub-Gaussian with variance proxy $\sigma^2 |t_2 - t_1|^\alpha$.
\end{proposition}

All the proofs are deferred to Appendix~\ref{App:proofs}.  This result shows that past observations provide noisy information about future responses, with a noise level increasing with the time lag. This structure is the key ingredient that allows us to leverage time-varying GP methods.

We now introduce standard regularity assumptions enabling GP regression.

\begin{assumption}\label{As:ResponseRegularity}
The response function $g$ belongs to an RKHS $\mathcal{H}_k$ induced by a positive definite kernel $k$, and $\|g\|_{\mathcal{H}_k} \leq B$.
\end{assumption}

\begin{assumption}\label{As:Lipschitz}
The reward function $f$ is $L_f$-Lipschitz in~$(x,y)$ with respect to $\|\cdot\|_1$.
\end{assumption}

\subsection{W-SparQ-GP-UCB for sequential games}

The key idea of our approach is to view the sequence $(g_t)_t$ as a time-varying function observed under a noise model with increasing variance:
\begin{align*}
\forall t_1 \leq t_2, \quad y_{t_1} = g_{t_2}(x_{t_1}) + \epsilon_{t_1,t_2},
\end{align*}
where $\epsilon_{t_1,t_2}$ has variance scaling as $\sigma^2(1 + (t_2 - t_1)^\alpha)$. This perspective allows us to directly apply the \textbf{W-SparQ-GP-UCB} framework~\citep{mauduit2025no}, where at some iterations, the learner is allowed to query the current opponent response at previously selected actions.
\begin{align*}
g_t(x_i), \quad i < t.
\end{align*}

These additional observations reduce uncertainty on the current response while keeping the overall number of queries sublinear.
To do so, time is partitioned into windows, and at the beginning of each window, a sparse subset of past actions is selected using a determinantal point process (DPP) mechanism~\citep{Kulesza_2012}. The learner then builds a GP posterior using both recent and additional observations, and selects actions according to the optimistic rule
\begin{align} \label{eqn:OptimisticRule}
x_t = \arg\max_{x \in \mathcal{X}} \max_{y \in [\mathrm{lcb}_t(x), \mathrm{ucb}_t(x)]} f(x,y).
\end{align}

The approach and the pseudo-code for this method are provided in Appendix~\ref{App:SeqGamesSparQ}.

\subsection{Dynamic regret analysis}

We now show that this approach achieves sublinear dynamic regret.

\begin{theorem}\label{Thm:BiLevelRegret}
Let~$0< \delta < 1$. Under Assumptions~\ref{As:ResponseRegularity}--\ref{As:Lipschitz}--\ref{As:opponentVar} on the learner reward~$f$ and opponent response~$g$, the proposed method with parameter~$0<\tilde{\alpha} < 1/3$ achieves a dynamic regret
\begin{align}\label{Eq:BLRegret}
R_T = \mathcal{O}\left(\sqrt{T^{\tilde{\alpha}+1} d \log^{d+1 }T\left(\log \frac{1}{\delta} + d^2  \log^{d+1}(\log T)\right)}\right).
\end{align}
\end{theorem}

The total number $N_T$ of additional queries  satisfies
\begin{align*}
 N_T = \mathcal{O}\big(T^{(\alpha-\tilde\alpha)/\alpha}\log^d T\big) = o(T),
\end{align*}
so the algorithm requires only minimal extra feedback beyond bandit observations, \textit{i.e.} a vanishing number of additional queries per iteration. As can be seen, the parameter $\tilde{\alpha}$ governs a tradeoff between the regret guarantees and the number of additional queries. Indeed, larger values of $\tilde{\alpha}$ lead to larger window sizes, and therefore fewer additional queries, but at the cost of a larger regret term, as shown in~\eqref{Eq:BLRegret}.

This result highlights a key trade-off: by relaxing the adversarial assumption and introducing controlled temporal variability, we can move from static to dynamic regret guarantees, enabling the learner to track the optimal action sequence over time.

\section{Bilevel optimization with unknown lower-level response} \label{Sec:bilevel}

We now extend the sequential game setting to the classical bilevel framework. 
In contrast with the previous section, the opponent's response is no longer an arbitrary function in an RKHS, but is defined implicitly as the solution of a parametric optimization problem. More precisely, at each step~$t$, the response $y_t^* \in \mathcal{Y} \subset \mathbb{R}^m$ satisfies
\[
y_t^*(x) \in \argmin_{y \in \mathcal{Y}} g_t(x,y),
\]
where $g_t : \mathcal{X} \times \mathcal{Y} \to \mathbb{R}$ is an unknown function. This extension introduces two main difficulties:
\begin{enumerate}
    \item The response is now \emph{vector-valued} and must be modeled using multi-output Gaussian processes.
    \item The regularity of the induced mapping 
     $ \tilde{g}_t : x \mapsto \argmin_{y \in \mathcal{Y}} g_t(x,y) $ 
    is not immediate and requires structural assumptions on $g_t$.
\end{enumerate}

The resulting bilevel problem reads
\begin{equation}\label{Eq:GeneralBiLevel}
\begin{aligned}
x_t^* \in \argmax_{x \in \mathcal{X}} \quad & f(x,y_t^*) \\
\text{s.t.} \quad & y_t^* \in \argmin_{y \in \mathcal{Y}} g_t(x,y),
\end{aligned}
\end{equation}
where the upper-level objective $f$ is known, while the lower-level function $g_t$ is unknown. Our goal is to extend the \textbf{W-SparQ-GP-UCB} methodology to this setting, and provide regret guarantees.

\subsection{Time-invariant bilevel problems} 

We first consider the stationary case where for all $t \geq 0$, $g_t \equiv g$:
\begin{equation}\label{Eq:StaticGeneralBiLevel}
\begin{aligned}
x^* \in \argmax_{x \in \mathcal{X}} \quad & f(x,y^*) \\
\text{s.t.} \quad & y^* \in \argmin_{y \in \mathcal{Y}} g(x,y).
\end{aligned}
\end{equation}

We assume access to noisy observations of $\tg$, that is, the response map that we aim to infer via GP regression
$\tilde{g}(x) := \argmin_{y \in \mathcal{Y}} g(x,y)$.

A central question is therefore: \emph{under which conditions on $g$ does $\tilde{g}$ belong to a suitable RKHS?}

Let $k$ be a positive definite kernel on $\mathbb{R}^d$. We denote by $\Hk_{k,\mathcal{X}}$ the RKHS restricted to~$\mathcal{X}$: $\Hk_{k,\mathcal{X}} := \{f|_{\mathcal{X}} : f \in \Hk_k(\mathbb{R}^d)\}$, equipped with the quotient norm. For vector-valued functions, we consider $\Hk_{k,\mathcal{X}}^m$ with norm $\|\tilde{g}\|_{\Hk_{k,\mathcal{X}}^m}^2 = \sum_{i=1}^m \|\tilde{g}_i\|_{\Hk_{k,\mathcal{X}}}^2$.

In the main text, results will be provided for the class of Matérn kernels. A similar analysis for the squared exponential (SE) kernel is conducted in Appendix~\ref{App:SEKernel}.

Let $l>0$ and $\nu>0$.  We consider Matérn kernels defined by
\begin{equation}\label{Eq:MaternKernel}
k_{\nu}(x,x') = \frac{2^{1-\nu}}{\Gamma(\nu)} \left( \frac{\|x-x'\|_2}{l} \right)^\nu K_\nu\left( \frac{\|x-x'\|_2}{l} \right),
\end{equation}
whose RKHS satisfies $\Hk_{k_\nu,\mathcal{X}} = H^{\nu + d/2}(\mathcal{X})$~\citep{wendland2004scattered}, where $H^s(\mathcal{Z})$ denotes the  Sobolev space of regularity~$s$ on $\mathcal{Z}$. We now  explicit  assumptions ensuring that $\tilde{g}$ is well-defined and regular.

\begin{assumption}\label{ass:matern}
Let $s := \nu + d/2$. The function $g$ satisfies:
\begin{enumerate}
    \item[\textup{(A1)}] \textbf{Uniform strong convexity:} there exists $\mu > 0$ such that
    \[
    \nabla^2_{yy} g(x,y) \succeq \mu I_m, \quad \forall (x,y) \in \mathcal{X} \times \mathcal{Y}.
    \]
    \item[\textup{(A2)}] \textbf{Sobolev regularity:}
    $g \in H^{s+1}(\mathcal{X} \times \mathcal{Y})$.
\end{enumerate}
\end{assumption}

\begin{proposition}[Argmin map regularity under Matérn assumptions]
\label{prop:matern}
Under Assumption~\ref{ass:matern}, the mapping $\tilde{g}$ satisfies:
\begin{enumerate}
    \item \textbf{Well-definedness:} $\tilde{g}(x)$ is uniquely defined for all $x \in \mathcal{X}$.
    \item \textbf{RKHS membership:} $\tilde{g} \in \Hk_{k_\nu,\mathcal{X}}^m$.
\end{enumerate}
\end{proposition}

\subsection{GP-UCB for time-invariant bilevel problems} \label{sec:tibi}

Under the previous assumptions, $\tilde{g} \in \Hk_{k,\mathcal{X}}^m$ with $\|\tilde{g}\| \leq B$. We can therefore model $\tilde{g}$ using a multi-output GP. Let $\boldsymbol{\mu}_t(x), \boldsymbol{\sigma}_t(x) \in \mathbb{R}^m$ denote posterior mean and standard deviation. Define confidence sets
$\tilde{\mathcal{U}}_{t}(x) = \left[\boldsymbol{\mu}_{t-1}(x) \pm \beta_t \boldsymbol{\sigma}_{t-1}(x)\right]$,
with
\begin{equation}\label{Eq:MOBeta}
\beta_t = \sqrt{2 \log \left( \frac{m}{\delta} \frac{|\Sigma_t + K_t|^{1/2}}{|K_t|^{1/2}} \right)} + B.
\end{equation}

The acquisition rule is
\begin{equation}\label{Eq:BiLevelRule}
x_t = \argmax_{x \in \mathcal{X}} \max_{y \in \tilde{\mathcal{U}}_t(x)} f(x,y).
\end{equation}

We refer to this algorithm as \textbf{GP-UCBL}.

\begin{theorem}\label{Thm:UCBLRegret}[Regret of GP-UCBL]
\label{thm:ucbl}
Take $0< \delta <1$ and some time horizon $T$. Let~$k$ be a Matérn and Assumption~\ref{ass:matern} hold. Assume $f$ is $L_f$-Lipschitz. Then, for $\beta_t$ defined in~\eqref{Eq:MOBeta}, with probability at least $1-\delta$, GP-UCBL achieves
\[
R_T = \mathcal{O}\left(m \sqrt{T \gamma_T \left(\gamma_T + \log \frac{1}{\delta}\right)}\right).
\]
\end{theorem}

\paragraph{Remark.}
We can formulate a version of Theorem~\ref{Thm:UCBLRegret} for the SE kernel using Ass.~\ref{ass:rbf}. For both Matérn with $\nu > \frac{d(d+1)}{2}$ and SE kernels, $\gamma_T = o\left(T^{1/2}\right)$, implying $R_T = o(T)$~\citep{srinivas2012information}.

\subsection{W-SparQ-GP-UCB for time-varying bilevel problems} \label{sec:tvbi}

We now extend the previous bilevel setting to the time-varying case and introduce \textbf{W-SparQ-BL}, a windowed sparse variant of \textbf{GP-UCBL} tailored to handle temporal variations in the lower-level problem. We consider the bilevel problem~\eqref{Eq:GeneralBiLevel}, where the sequence $(g_t)_{t \geq 1}$ evolves over time. At each round $t$, we observe noisy evaluations of the lower-level function, which induce observations of the response map
$y^i_t = \tilde{g}^i_t(x_t) + \varepsilon^i_t$, where $\tilde{g}_t(x) := \argmin_{y \in \mathcal{Y}} g_t(x,y)$. As in the core W-SparQ framework (see Appendix~\ref{App:SeqGamesSparQ}), we assume a time-increasing noise model:
\[
\varepsilon^i_t \sim \text{subG}(\sigma_t^2), \qquad \sigma_t^2 \underset{t \to +\infty}{\longrightarrow} +\infty,
\]
which captures the fact that past observations become progressively less informative about the current function. To control this effect, we introduce the same windowing and querying mechanism as in W-SparQ-GP-UCB to keep only recent observations, allowing to maintain the variance proxy level~$\mathcal{O}(t^{\tilde\alpha})$.

We work with a Matérn-$\nu$ kernel $k$ and seek conditions ensuring that all response maps $(\tilde{g}_t)$ remain in a common RKHS ball. Analysis elements for SE kernel in the time-varying framework are given in Appendix~\ref{App:SETV}.

\begin{assumption}[Uniform regularity of $(g_t)$]
\label{ass:MaternRKHSBall}
Let $s := \nu + d/2$. The sequence $(g_t)$ satisfies:
\begin{enumerate}
    \item[\textup{(A1)}] \textbf{Uniform strong convexity:}
    $\nabla^2_{yy} g_t(x,y) \succeq \mu I_m, \quad \forall t, \; (x,y) \in \mathcal{X} \times \mathcal{Y}.$
    \item[\textup{(A2)}] \textbf{Uniform Sobolev control:}
    $\sup_{t \geq 1} \|g_t\|_{H^{s+1}(\mathcal{X} \times \mathcal{Y})} \leq M$.
\end{enumerate}
\end{assumption}

Under these assumptions, each $\tilde{g}_t$ is well-defined and inherits Sobolev regularity.

\begin{proposition}[Uniform RKHS control of $(\tilde{g}_t)$]
\label{prop:uniform_rkhs}
Under Assumption~\ref{ass:MaternRKHSBall}, there exists $B < +\infty$ such that $\sup_{t \geq 1} \|\tilde{g}_t\|_{\Hk_{k,\mathcal{X}}^m} \leq B$.
\end{proposition}

We now quantify how fast the response maps may vary over time.

\begin{proposition}[Temporal variation bound]
\label{Prop:BiLevelTemporalVar}
Assume Proposition~\ref{prop:uniform_rkhs} holds and that $k(x,x) \leq M_k$. Then for all $n \geq 1$,
\[
\|\tilde{g}_{t+n}^i - \tilde{g}_t^i\|_\infty 
\;\leq\; 2 n B \sqrt{M_k}, \qquad \forall i=1,\dots,m.
\]
\end{proposition}


The  \textbf{W-SparQ-BL} algorithm is described in Algorithm~\ref{AppAlg:W-SparQ-BL}. We apply \textbf{W-SparQ-GP-UCB} \emph{element-wise} to estimate each coordinate of $\tilde{g}_t$. At each time step:
\begin{itemize}
    \item a windowed dataset $(\Xr_t, \Yr_t)$ is maintained,
    \item sparse pseudo-observations are added at the beginning of each window,
    \item a multi-output GP posterior $(\boldsymbol{\mu}_t, \boldsymbol{\sigma}_t)$ is constructed.
\end{itemize}
Finally, the next action is selected via
\begin{equation}\label{Eq:NextIt}
x_{t+1} = \argmax_{x \in \mathcal{X}} \max_{y \in \tilde{\mathcal{U}}_t(x)} f(x,y).
\end{equation}

\begin{algorithm}
	\caption{W-SparQ-BL}\label{AppAlg:W-SparQ-BL}
	\begin{algorithmic}[1]
		\Require Action set $\mathcal{X}$, kernel $k$, parameters $\{\beta_t\}_{t\geq1},\tilde\alpha$
        \Ensure Sequence of selected actions $\{x_t\}_{t\in\mathbb{N}}$
        \For{$t = 1,\dots$}
			\If{Beginning of a window}
			\State Perform sparse inference on~$X_t$ to obtain~$\Xs_t$
			\State Get updated outputs~$\Ys_t$
			\State Form posterior on~$\tg_{t+1}$ by Bayesian update on~$(\Xs_t, \Ys_t)$ 
			\State Initialize~$(\Xr_t,\Yr_t) \leftarrow (\Xs_t, \Ys_t)$
			\Else
			\State Append~$(x_t, y_t)$ to~$(\Xr_t, \Yr_t)$
			\State Form posterior by Bayesian update on~$(\Xr_t, \Yr_t)$
			\EndIf
			\State Select $x_{t+1}$ via~\eqref{Eq:NextIt}
			\EndFor
	\end{algorithmic}
\end{algorithm}

Compared to the SE kernel, the Matérn information gain grows faster, which impacts the parameters of the algorithm. The noise proxy must satisfy $
    \tilde{\sigma}_t^2 = \mathcal{O}\left(t^{\tilde{\alpha}}\right)$, $\tilde{\alpha} < \frac{2 \nu - d(d+1)}{4 \nu + 2d(d+1)},
    $, and each window requires $\mathcal{O}\left(T^{\frac{2d}{2\nu-d}}\right)$ 
    additional queries. These scalings follow from adapting the spectral analysis of~\cite{mauduit2025time} to Matérn kernels.

\begin{theorem}[Regret of W-SparQ-BL]
\label{Thm:WSPARQBLRegret}
Under Assumption~\ref{ass:MaternRKHSBall}, let $\tilde{\alpha}$ satisfy
$\tilde{\alpha} < \frac{2 \nu - d(d+1)}{4 \nu + 2d(d+1)}$.
Assume $f$ is $L_f$-Lipschitz. Then, with probability at least $1-\delta$, by performing~$Q_{t_k}$ additional observations at each window start~$t_k$, W-SparQ-BL achieves
\[
R_T = \mathcal{O}\left(
m \sqrt{T^{\tilde{\alpha}+1} \gamma_T \left(\gamma_{Q_T} + \log \frac{1}{\delta}\right)}
\right).
\]
\end{theorem}

\paragraph{Remark.}
The choice of $\tilde{\alpha}$ ensures $R_T = o(T)$. However, the total number of additional queries scales as $N_T = T^{1-\tilde{\alpha}} \cdot T^{\frac{2d}{2\nu-d}}$, which can be superlinear depending on $(\nu,d)$, highlighting a trade-off between adaptivity and sample complexity. The number of additional queries~$Q_{t_k}$ can be explicitly computed: for SE kernel $Q_{t_k} = \mathcal{O}(\log^d(t_k))$ and Matérn-$\nu$ kernels $Q_{t_k} = \mathcal{O}\left(t_k^{\frac{2d}{2\nu-d}}\right)$.

\section{Numerical experiments} \label{Sec:numerics}

\paragraph{Time-varying sequential game.}

We evaluate the sequential game variant of W-SparQ-GP-UCB described in Algorithm~\ref{Appalg:SeqW-SparQ-GP-UCB} on the same experimental setting as in Section~4.1 of~\cite{sessa2020learning}, which consists of a traffic routing task. In this problem, a learner (e.g., a routing agent) proposes routing plans to a population of users (e.g., vehicles), with the objective of guiding each unit to its destination while keeping congestion low over the network. Details for the dataset and experimental setting are given in Appendix~\ref{App:ExpSeq}. To enhance temporal variations induced by the changes in opponents' types, we introduce an additional drift in the congestion over the edges. At each step, we impose a random congestion drift, which can be interpreted as a simplified model of external factors such as road works, accidents, or rush-hour effects.

\begin{figure}[H]
	\centering
	\includegraphics[scale=0.4]{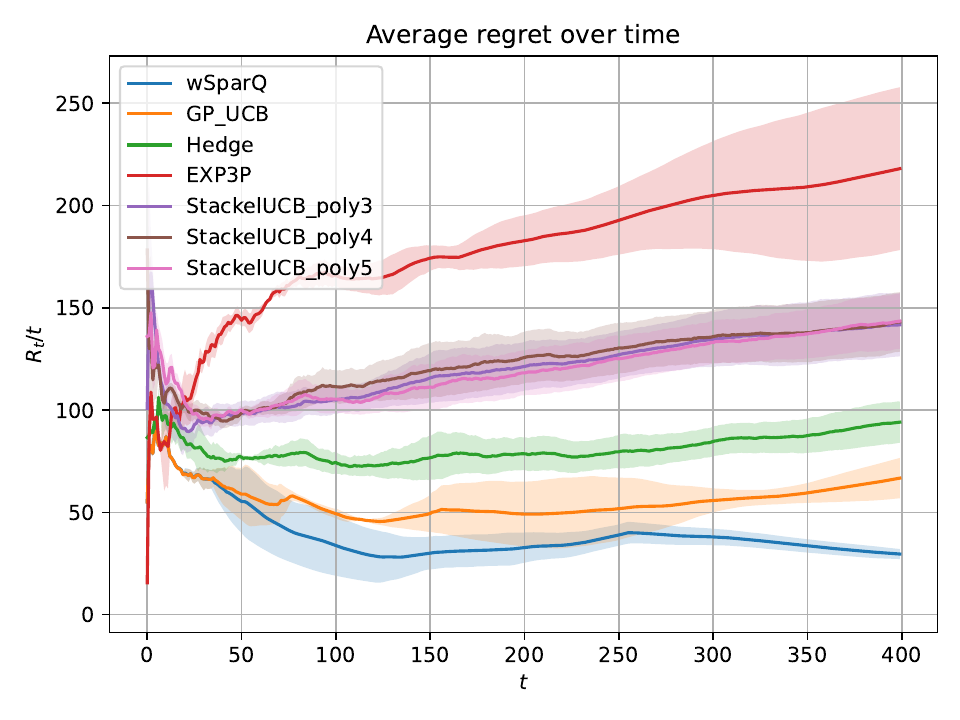}
	\caption{Average dynamic regret under  congestion drift.}
	\label{Fig:BiLevelDrift}
\end{figure}

Figure~\ref{Fig:BiLevelDrift} reports the results in this setting, where we compare our method, denoted wSparQ, to StackelUCB with different polynomial kernel degrees, as well as multiplicative weights algorithms, namely HEDGE~\citep{freund1997decision} and EXP3~\citep{auer2002nonstochastic}. We observe that wSparQ outperforms the other methods and exhibits asymptotically decreasing regret. Notably, HEDGE, despite having access to full information, struggles to adapt to temporal changes. This behavior can be attributed to its long-term memory: standard HEDGE aggregates past observations uniformly, which hinders adaptation to non-stationary environments. Incorporating forgetting mechanisms or windowing strategies could potentially improve its performance in this setting.

It is important to note that the Sioux-Falls dataset poses several challenges. First, the opponent’s response does not strictly satisfy Assumption~\ref{As:opponentVar}. More importantly, the input dimension is very high ($d=76$), as StackelUCB was originally designed for discrete action spaces. Since theoretical guarantees scale poorly with dimension, both in terms of regret bounds and required additional queries, this setting lies well outside the ideal regime of our analysis.
Despite these obstacles, wSparQ adapts well to high-dimensional data and achieves strong empirical performance using far fewer additional queries than predicted by theory. This suggests that the method remains practically applicable beyond the assumptions under which its theoretical guarantees are derived.

\paragraph{Time-varying bilevel optimization.} 
We evaluate \textbf{W-SparQ-BL} against \textbf{GP-UCBL} on three synthetic bilevel settings over the domain $\mathcal{X}=[0,1]$, designed to reflect increasing levels of temporal variability in the lower-level response. We use a Matérn kernel with smoothness parameter $\nu=1.5$, corresponding to twice differentiable response functions, in line with the Sobolev-type assumptions used in our theoretical analysis. See Appendix~\ref{App:ExpBL} for details on the functions. For both methods, we plot the mean and standard deviation (across $10$ realizations) of the cumulative regret over $T=300$ iterations.

\begin{figure}[H]
\centering
\begin{minipage}{0.32\textwidth}
    \centering
    \includegraphics[width=\textwidth]{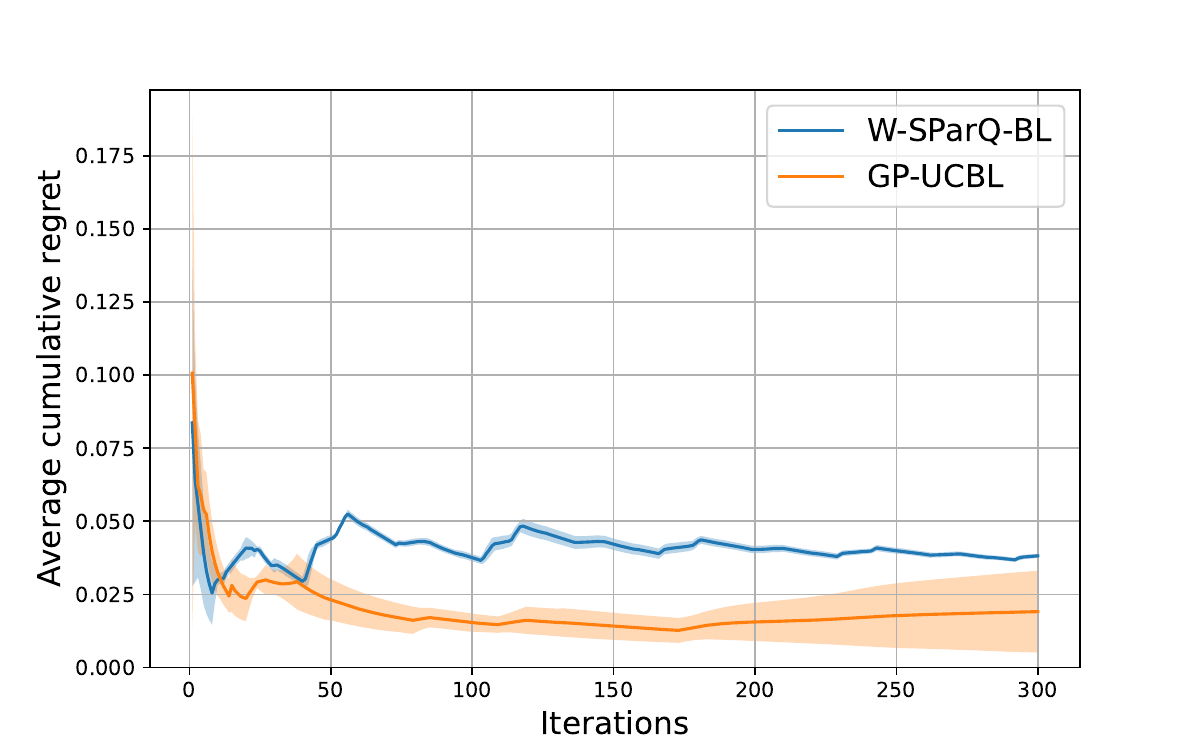}
    \caption*{Stationary}
\end{minipage}
\hfill
\begin{minipage}{0.32\textwidth}
    \centering
    \includegraphics[width=\textwidth]{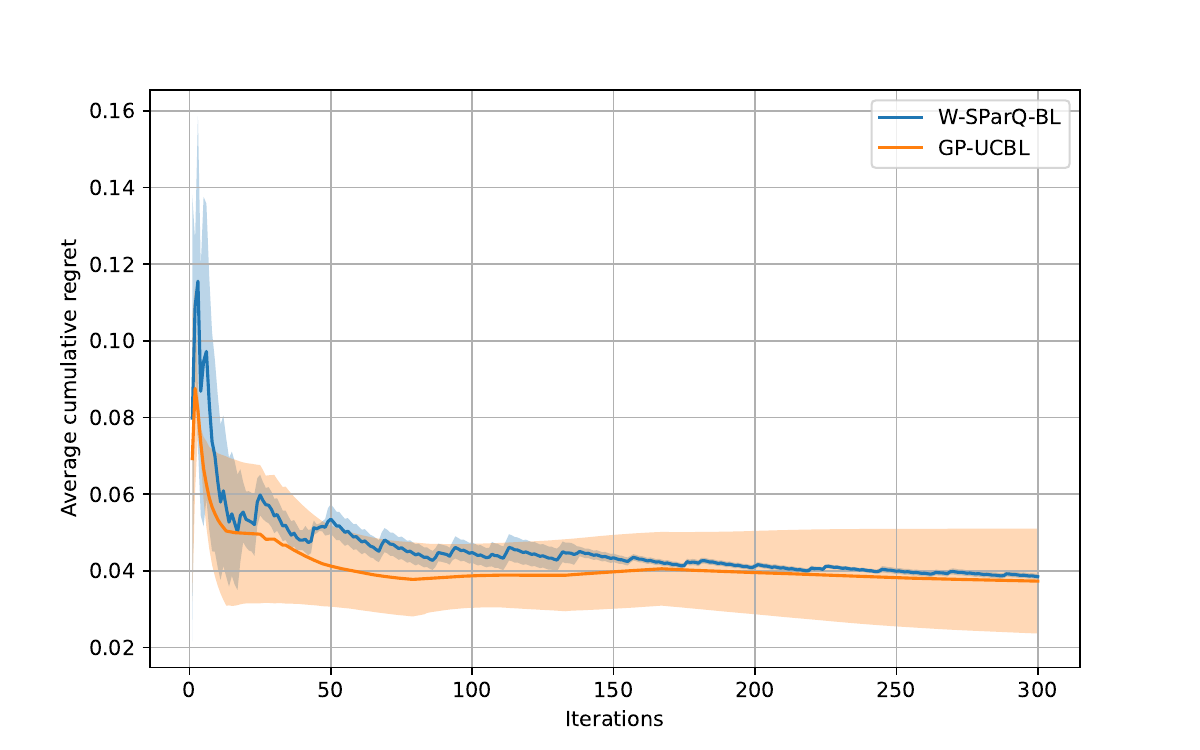}
    \caption*{Moderate variation}
\end{minipage}
\hfill
\begin{minipage}{0.32\textwidth}
    \centering
    \includegraphics[width=\textwidth]{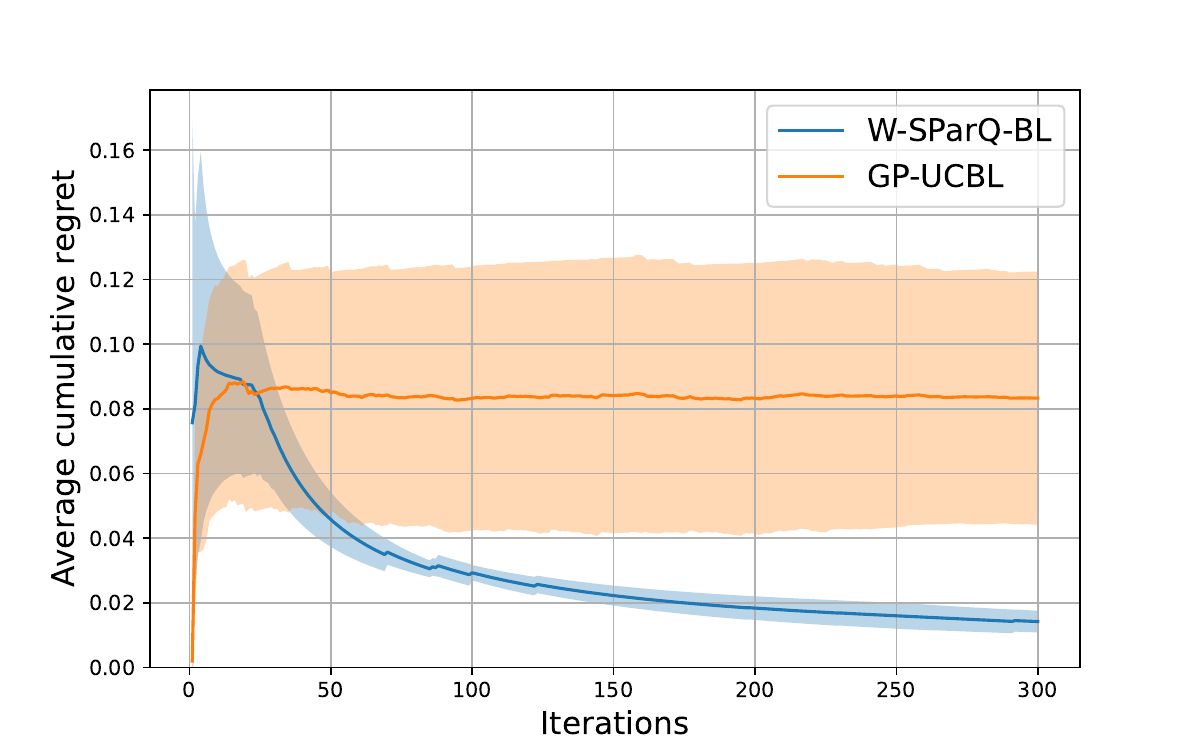}
    \caption*{Fast variation}
\end{minipage}
\caption{Average cumulative regret of \textbf{W-SparQ-BL} and \textbf{GP-UCBL} under increasing levels of temporal variability.}
\label{fig:experiments}
\end{figure}

In the stationary and moderate variation regimes, both methods achieve comparable performance, as expected since temporal adaptation is not required. We can even note that GP-UCBL slightly outperforms W-SparQ-BL, likely due to the noise model of the latter that favors unnecessary exploration in these settings. In the fast variation setting, W-SparQ-BL consistently outperforms GP-UCBL that suffers from outdated information. We might also note that the inter-runs variance of W-SparQ-BL is lower than the one of GP-UCBL, suggesting that the method is more robust to changes in the dataset.

\section{Conclusion}

We introduced \textbf{W-SparQ-BL}, a novel approach for Bayesian optimization of bilevel problems that combines windowing and sparse querying to handle time-varying settings.  In a first configuration, where the lower-level can be directly modeled as a response function in an RKHS, our method recovers guarantees comparable to standard single-level Bayesian optimization, both in terms of dynamic regret and number of additional queries. We then considered the more general and challenging bilevel setting, where the lower-level response is defined implicitly through an $\argmin$ mapping. In this case, ensuring that the induced mapping lies in an RKHS requires additional regularity assumptions on the lower-level objective. This step is critical to justify the use of Gaussian Process regression as a surrogate model.

Our analysis reveals a fundamental trade-off in GP-based bilevel optimization. On the one hand, strong regularity assumptions enable fast convergence rates and require only a limited number of additional queries. On the other hand, under weaker and more realistic assumptions, both the regret rates and, more importantly, the number of additional queries increase significantly.  These results highlight intrinsic limitations of kernel-based approaches in hierarchical optimization problems, and suggest that the statistical complexity of learning the lower-level structure cannot be decoupled from the optimization task. Future work directions include extending the analysis to settings where both the upper and lower level functions are unknown via joint GP regression.

\section*{Acknowledgement}
The work is partly supported by Hi! Paris and ANR/France 2030 program (ANR-23-IACL-0005) as well as the ANR project ANR-23-CE48-0011-01.

\bibliographystyle{plainnat}
\bibliography{references}


\newpage

\appendix


In this supplementary material, we provide additional technical elements, namely a reminder of the W-SparQ methodology for time-varying optimization (Appendix~\ref{App:SeqGamesSparQ}), additional results for the squared exponential kernel (Appendix~\ref{App:SEKernel}), details on the experimental setups (Appendix~\ref{App:Exp}), and proofs of all the theoretical results (Appendix~\ref{App:proofs}). 

\section{W-SparQ methodology}\label{App:SeqGamesSparQ}

In this section, we provide a self-contained presentation of the \textbf{W-SparQ-GP-UCB} methodology and explain how it can be adapted to the sequential games setting considered in this work.

\subsection{W-SparQ-GP-UCB for time-varying Bayesian optimization}

We begin by recalling the core ideas underlying \textbf{W-SparQ-GP-UCB}. A more detailed presentation can be found in~\cite{mauduit2025no}.
We consider a sequence of time-varying reward functions $(f_t)_{t \geq 1}$ defined over a domain $\mathcal{X}$, and aim to solve
\[
\underset{x \in \mathcal{X}}{\argmax} \, f_t(x),
\]
based on noisy observations
\[
\forall t, \quad y_t = f_t(x_t) + \epsilon_t, \qquad \epsilon_t \sim \mathcal{N}(0, \sigma^2).
\]

The following assumptions are made on the function's spatial and temporal regularity.

\begin{assumption}[Spatial regularity]\label{as.wsparq1}
For all $t \in \mathbb{N}$, the function $f_t$ belongs to the RKHS $\mathcal{H}_k$ associated with a bounded kernel $k$, and satisfies $\|f_t\|_{\mathcal{H}_k} \leq B$ almost surely.
\end{assumption}

\begin{assumption}[Temporal regularity]\label{as.wsparq2}
There exists $\alpha \geq 0$ such that, for all $t_1 \leq t_2$ and all $x \in \mathcal{X}$,
\[
f_{t_1}(x) - f_{t_2}(x) \sim \mathrm{subG}\big(\sigma^2 (t_2 - t_1)^\alpha\big).
\]
\end{assumption}
In particular, Assumption~\ref{as.wsparq2} captures the fact that observations become progressively less informative about future functions as the time gap increases.

A key idea of \textbf{W-SparQ-GP-UCB} is to reinterpret past observations as noisy measurements of the current function by explicitly modeling this temporal drift as additional noise. This leads to the equivalent observation model
\begin{equation}\label{Eq:UIModel}
\forall t_1 \leq t_2, \quad y_{t_1} = f_{t_2}(x_{t_1}) + \epsilon_{t_1, t_2}, 
\qquad \epsilon_{t_1, t_2} \sim \mathrm{subG}\big(\sigma^2(1 + (t_2 - t_1)^\alpha)\big).
\end{equation}
This model enables the use of heteroscedastic GP regression to estimate $f_{T+1}$ from past observations $(X_T, Y_T)$, using a noise covariance matrix
\[
\Sigma_T = \mathrm{diag}\big(\sigma^2(1 + n^\alpha)\big)_{n=0}^{T-1}.
\]

Since the noise variance grows as $\mathcal{O}(T^\alpha)$, older observations quickly become unreliable. As a consequence, directly applying standard GP-UCB leads to poor performance and, in general, fails to achieve sublinear regret.

A first approach consists in refreshing the dataset at every iteration by selecting a representative subset of past inputs via a determinantal point process (DPP)~\citep{Kulesza_2012}, and querying their current values. This strategy allows one to approximate the posterior obtained from fully fresh observations with bounded noise, and recover standard regret guarantees. However, performing such updates at every iteration is unnecessary.

Indeed, a key observation is that there exists a threshold $\tilde{\alpha}$ (depending on the kernel through its information gain) such that if the effective noise grows as $\mathcal{O}(t^{\tilde{\alpha}})$, then GP-UCB achieves no-regret.
This motivates partitioning time into non-overlapping windows $\mathcal{W}_j = \{t_j, \ldots, t_{j+1}-1\}$ satisfying
\begin{equation}\label{eq:window-size}
t_j^{\tilde\alpha/\alpha} < t_{j+1} - t_j \leq t_j^{\tilde\alpha/\alpha} + 1.
\end{equation}
With this construction, for any $t \in \mathcal{W}_j$, observations collected within the same window have an effective noise variance bounded by $\mathcal{O}(t^{\tilde{\alpha}})$. Therefore, it is sufficient to perform the costly DPP-based refresh only at the \emph{beginning} of each window.

The resulting method, \textbf{W-SparQ-GP-UCB}, is described in Algorithm~\ref{alg:W-SparQ-GP-UCB}. It alternates between:
\begin{itemize}
\item sparse updates at the beginning of each window,
\item standard GP-UCB updates within each window.
\end{itemize}

\begin{theorem}[\citep{mauduit2025no} Regret bound for \textbf{W-SparQ-GP-UCB}]\label{Thm:regretW-SparQ}
		Let~$(f_t)_t$ be a sequence of reward functions and~$y_t=f_t(x_t)+\epsilon_t$ with~$\epsilon_t\sim\mathcal{N}(0,\sigma^2)$. Under Assumptions~\ref{as.wsparq1}--\ref{as.wsparq2} with parameter~$\alpha$ and a squared exponential kernel~$k$, let~$(x_t)_{t=1}^T$ be the decisions of Algorithm~\ref{alg:W-SparQ-GP-UCB}, for~$0<\tilde\alpha<1/3$ chosen by the user. For any~$0<\delta\le1$, define
         \[\beta_t = \sqrt{2 \log \left(  \frac{|\Sigma_t + K_t|^{1/2}}{\delta|K_t|^{1/2}} \right)} + B.
        \]
         Then, with probability at least $1-\delta$, by querying the expert~$Q_{t_k}=\mathcal{O}(\log^d t_k)$ times at each window starting at~$t_k$, the cumulative dynamic regret satisfies
		\begin{equation}\label{Eq:WSparQUpper}
			R_T = \mathcal{O}\left(\sqrt{T^{\tilde{\alpha}+1} d \log^{d+1 }T\left(\log \frac{1}{\delta} + d^2\log^{d+1}\left(\log T\right)\right)}\right).
		\end{equation}
		Moreover, the average number of expert queries per step vanishes:
		\begin{align*}
		\frac{N_T}{T} = \mathcal{O}\big(T^{-\tilde\alpha/\alpha}\log^d T\big) = o(1).
		\end{align*}
	\end{theorem}

\begin{algorithm}
		\caption{W-SparQ-GP-UCB}\label{alg:W-SparQ-GP-UCB}
		\begin{algorithmic}[1]
			\Require Domain \(\mathcal{X}\), kernel \(k\), time horizon \(T\), parameters \(\tilde\alpha,\alpha\)
			\Ensure Sequence of selected actions $\{x_t\}_{t=1}^T$
			\For{$t = 1,\dots,T$}
			\If{Beginning of a window}
			\State Perform sparse inference on~$X_t$ to obtain~$\Xs_t$
			\State Get updated outputs~$\Ys_t$
			\State Form posterior by Bayesian update on~$(\Xs_t,\Ys_t)$
			\State Initialize~$(\Xr_t,\Yr_t) \leftarrow (\Xs_t,\Ys_t)$
			\Else
			\State Append~$(x_t,y_t)$ to~$(\Xr_t,\Yr_t)$
			\State Form posterior by Bayesian update on~$(\Xr_t,\Yr_t)$
			\EndIf
			\State Select \(x_{t+1}=\arg\max_{x\in\mathcal{X}}\; \mu_t(x)+\beta_{t+1}\sigma_t(x)\)
			\EndFor
		\end{algorithmic}
	\end{algorithm}

\subsection{Application to sequential games}

We now show how this methodology extends to the sequential games setting.
The main insight is that the sequence of response functions $(g_t)_t$ induced by the opponent can be treated as a time-varying function satisfying a noise growth model similar to~\eqref{Eq:UIModel}. As a result, the sequential game reduces to a time-varying GP regression problem.

We partition time into windows as before and construct GP confidence intervals for the response function. At each iteration, the learner selects actions according to the optimistic bilevel rule
\begin{equation}\label{AppEq:LeaderNextIt}
x_t = \arg\max_{x\in\mathcal{X}}\;\max_{y\in[\mathrm{lcb}_{t}(x),\mathrm{ucb}_{t}(x)]} r(x,y).
\end{equation}

At the beginning of each window, $Q_t$ additional queries are performed on a sparse subset $\Xs_{t_j}$ selected via a DPP:
\[
Q_{t_j} =
\begin{cases}
\mathcal{O}(\log^d t_j)\,, & \text{for the SE kernel}, \\
\mathcal{O}\left(T^{\frac{2d}{2\nu-d}}\right), & \text{for the $\nu$-Matérn kernel}.
\end{cases}
\]

These additional observations significantly reduce uncertainty on the current response while preserving controlled query complexity.
The method builds high-confidence bounds on the unknown response function, which are then propagated through the reward via an optimistic construction. This enables the learner to balance exploration (learning the opponent’s behavior) and exploitation (maximizing reward). Algorithm~\ref{Appalg:SeqW-SparQ-GP-UCB} summarizes the main steps of the method.

\begin{algorithm}[H]
	\caption{W-SparQ-GP-UCB for sequential games}\label{Appalg:SeqW-SparQ-GP-UCB}
	\begin{algorithmic}[1]
		\Require Action set $\mathcal{X}$, kernel $k$, parameters $\{\beta_t\}_{t\geq1},\tilde\alpha,\alpha$
        \Ensure Sequence of selected actions $\{x_t\}_{t\in\mathbb{N}}$
        \For{$t = 1,\dots$}
			\If{Beginning of a window}
			\State Perform sparse inference on~$X_t$ to obtain~$\Xs_t$
            \State Update~$\Thetas_t = \{\theta_t\}^{Q_t}$
			\State Get updated outputs~$\Ys_t$
			\State Form posterior by Bayesian update on~$(\Xs_t, \Thetas_t, \Ys_t)$
			\State Compute $\tilde{r}_t(\cdot,\theta_t)$
			\State Initialize~$(\Xr_t,\Thetar_t,\Yr_t) \leftarrow (\Xs_t, \Thetas_t, \Ys_t)$
			\Else
			\State Append~$(x_t, \theta_t, y_t)$ to~$(\Xr_t,\Thetar_t, \Yr_t)$
			\State Form posterior by Bayesian update on~$(\Xr_t, \Thetar_t, \Yr_t)$
			\EndIf
			\State Select \(x_{t+1}=\arg\max_{x\in\Dset}\; \tilde{r}_t(x,\theta_t) \)
			\EndFor
	\end{algorithmic}
\end{algorithm}

This construction directly leads to the regret guarantees established in the main text.

\section{Squared Exponential kernel for bilevel Bayesian optimization}\label{App:SEKernel}

In this appendix, we discuss assumptions on the reponse~$g$ to obtain argmin mappings within the RKHS of squared exponential kernels, and compare them to those obtained for the class of Matérn kernels, both in the time-invariant and time-varying settings.

\subsection{Squared exponential kernel: assumptions in the time-invariant setting}\label{App:SETI}

The situation is fundamentally different for the squared exponential kernel. 
Its RKHS consists of restrictions of entire functions satisfying a Bargmann--Fock growth condition~\citep{steinwart2006explicit}. As a consequence, ensuring $\tilde{g} \in \Hk_{k,\mathcal{X}}^m$ requires global analyticity of $g$ in $x$, as well as the absence of singularities in its complex extension.

In particular, unlike the Matérn case, local Sobolev regularity is not sufficient. This makes the SE kernel significantly more restrictive for modeling argmin mappings, despite its appealing smoothness properties.
%
This dichotomy will be reflected in our regret bounds: while both kernels yield sublinear regret, the Matérn setting provides a more robust and verifiable modeling framework. We formalize this intuition in the following assumption.

\begin{assumption}\label{ass:rbf}
\begin{enumerate}
    \item[\textup{(B1)}] For every $y \in \mathcal{Y}$, the mapping $x \mapsto g(x,y)$ extends to an entire function on $\mathbb{C}^d$.
    \item[\textup{(B2)}] These extensions satisfy a uniform growth condition of the form
    \[
    |g(z,y)| \leq C \exp\big(c \|z\|^2\big), \quad \forall z \in \mathbb{C}^d, \; y \in \mathcal{Y},
    \]
    for some constants $C,c > 0$.
    \item[\textup{(B3)}] $g$ satisfies the strong convexity condition (A1).
\end{enumerate}
\end{assumption}
Under these assumptions, the analytic implicit function theorem ensures that $\tilde{g}$ extends holomorphically to $\mathbb{C}^d$ with controlled growth, implying $\tilde{g} \in \Hk_{k_l,\mathcal{X}}^m$.

\subsection{Squared exponential kernel: assumptions in the time-varying setting}\label{App:SETV}

As in the time-invariant case, extending the analysis to SE kernels requires global analyticity assumptions on $(g_t)$.  Ensuring $\tilde{g}_t \in \Hk_{k_l,\mathcal{X}}^m$ uniformly in $t$ requires:
\begin{itemize}
    \item uniform analyticity of $g_t$ in $x$,
    \item a uniform lower bound on the radius of analyticity,
    \item and compatibility with the Bargmann--Fock growth condition.
\end{itemize}

These conditions are significantly stronger than the Sobolev assumptions used for Matérn kernels, but in return yield faster information gain decay and improved regret rates. Hence, Matérn kernels provide a more flexible modeling framework for time-varying bilevel problems, while SE kernels offer faster rates at the cost of strong global structural assumptions.

\subsection{Comparison of Matérn and SE kernels}\label{App:SEComp}

We finally display in Table~\ref{tab:tableComparison} the main comparison elements for Matérn and SE kernels, illustrating the natural tradeoff emerging (regularity assumptions vs. algorithmic efficiency)
\begin{table}[H]
\centering
\caption{
Comparison of the assumptions and theoretical guarantees obtained for the squared exponential (SE) and Matérn-$\nu$ kernels. 
Here, $\gamma_T$ denotes the information gain, $Q_T$ the number of additional queries, and $\tilde{\alpha}$ the admissible growth exponent threshold.
}
\renewcommand{\arraystretch}{1.25}
\begin{tabular}{lcc}
\hline
 & SE kernel & Matérn-$\nu$ kernel \\
\hline
Assumptions on $g_t$
&
\begin{minipage}[t]{0.33\textwidth}
\vspace{2pt}
\begin{itemize}\setlength\itemsep{0.2em}
    \item uniform strong convexity in $y$
    \item analyticity in $x$
    \item uniform analyticity radius
    \item Bargmann--Fock growth condition
\end{itemize}
\vspace{2pt}
\end{minipage}
&
\begin{minipage}[t]{0.33\textwidth}
\vspace{2pt}
\begin{itemize}\setlength\itemsep{0.2em}
    \item uniform strong convexity in $y$
    \item Sobolev regularity of order $\nu+d/2+1$
    \item uniform Sobolev bounds
\end{itemize}
\vspace{2pt}
\end{minipage}
\\
\hline
$\gamma_T$
& $d\,\log^{d+1}(T)$
& $T^{\frac{d(d+1)}{2\nu+d(d+1)}}$
\\[0.4em]
\hline
$Q_T$
& $\log^d(T)$
& $T^{\frac{2d}{2\nu-d}}$
\\[0.4em]
\hline
Admissible $\tilde{\alpha}$
& $0 < \tilde{\alpha} < \frac{1}{3}$
& $\tilde{\alpha} < \frac{2\nu-d(d+1)}{4\nu+2d(d+1)}$
\\
\hline
\end{tabular}
\label{tab:tableComparison}
\end{table}

\section{Experimental settings}\label{App:Exp}

In this sections, we provide a more in-depth description of the datasets, as well as additional experiments on the traffic routing task.

\subsection{Time-varying sequential games}\label{App:ExpSeq}

We use the Sioux-Falls road traffic network~\citep{leblanc1975efficient}, composed of~$24$ nodes and~$76$ edges. At each iteration, the learner routes~$300$ units between a fixed origin and destination, represented in green and blue in Figure~\ref{Fig:SiouxFalls}.

\begin{figure}
	\centering
	\includegraphics[scale=1]{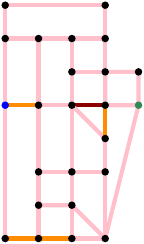}
	\caption{Schematic representation of the Sioux-Falls road graph.}
	\label{Fig:SiouxFalls}
\end{figure}

At each step~$t$, the learner proposes a routing plan~$x_t \in \mathcal{X} = \mathbb{N}^{|\mathrm{E}|}$, where~$\mathrm{E}$ denotes the set of edges. Each component~$x_t[i]$ corresponds to the number of vehicles routed through edge~$i \in \mathrm{E}$. The opponent’s type~$\theta_t \in \mathbb{N}^{552}$ represents the users’ preference profile at time~$t$: each coordinate encodes the number of users traveling between a given origin--destination pair in the network (there are~$24 \times 23 = 552$ such pairs).

After observing~$x_t$, users select their routes according to their preferences and the learner’s proposal, resulting in congestion over the edges of the network. Congestion levels are computed using models similar to those in~\cite{leblanc1975efficient} and illustrated as an example in Figure~\ref{Fig:SiouxFalls} using different shades of red. As in~\cite{sessa2020learning}, we define
\[
y_t = b(x_t, \theta_t) + \epsilon_t \in \mathbb{R}_+,
\]

where~$b(\cdot,\cdot)$ denotes the (unknown) average congestion function and~$\epsilon_t$ is observation noise. This is the quantity we aim to learn from noisy observations.

The learner’s reward is defined as
\[
r(x_t, y_t) = g(x_t) - \kappa\, y_t,
\]
where~$g(x_t)$ is the total number of units successfully routed to the destination node and~$\kappa>0$ is a trade-off parameter that balances between both objectives of the learner, which are routing as many units as possible to the destination and maintaining low congestion levels. We initially follow the experimental protocol described in Appendix~E of~\cite{sessa2020learning}, except that we compare the average dynamic cumulative regret~\eqref{Eq:DynamicRegret} across methods: note that Sessa et al.\ target a different regret notion (defined in~\eqref{Eq:SessaRegret}), which compares performance to a single best action in hindsight.

In our framework, the number of additional queries at the beginning of each window scales as~$Q_t = \mathcal{O}(\log^d t)$, where~$d=76$ is the dimension of~$\mathcal{X}$. This dependence would result in an impractically large number of additional queries. To manage this issue, we first consider the effective dimension of the dataset explaining~$95\%$ of the variance, which gives~$d_{\text{eff}}=7$. However, even this choice leads to a large amount of additional information. In practice, we therefore set~$d=2$, which empirically results in similar performance while drastically reducing the number of additional queries. The corresponding results are shown in Figure~\ref{Fig:BiLevelNoDrift}.

\begin{figure}
     \centering
     \begin{subfigure}[b]{0.32\textwidth}
         \centering
         \includegraphics[width=\textwidth]{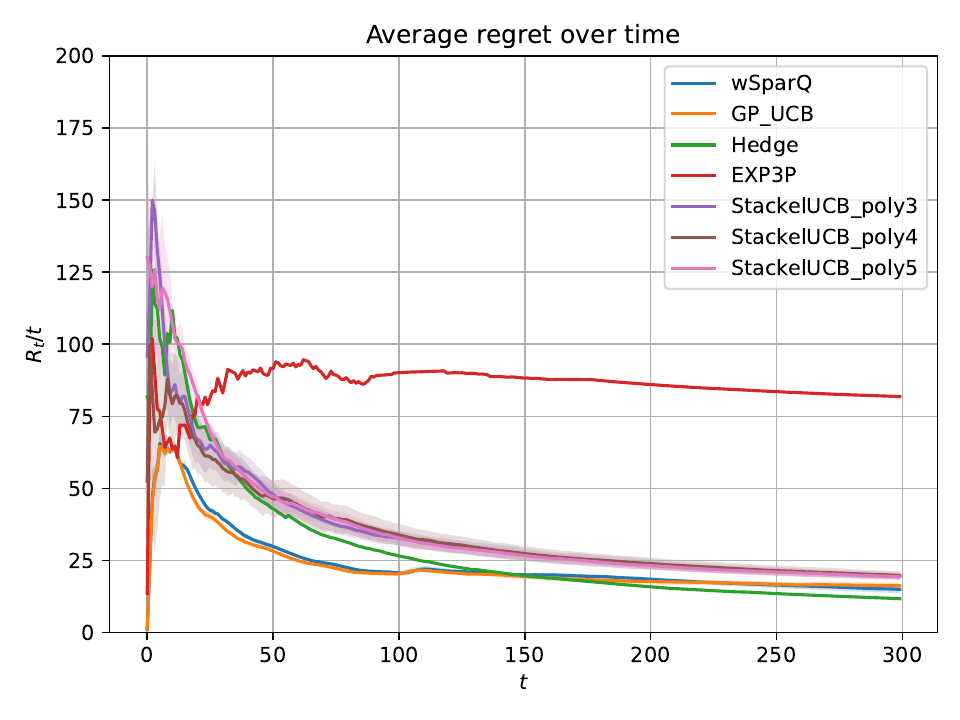}
         \caption{$\beta_t = 0.5,\; d=7$}
         \label{Fig:NoDriftFixedBetad7}
     \end{subfigure}
     \hfill
     \begin{subfigure}[b]{0.32\textwidth}
         \centering
         \includegraphics[width=\textwidth]{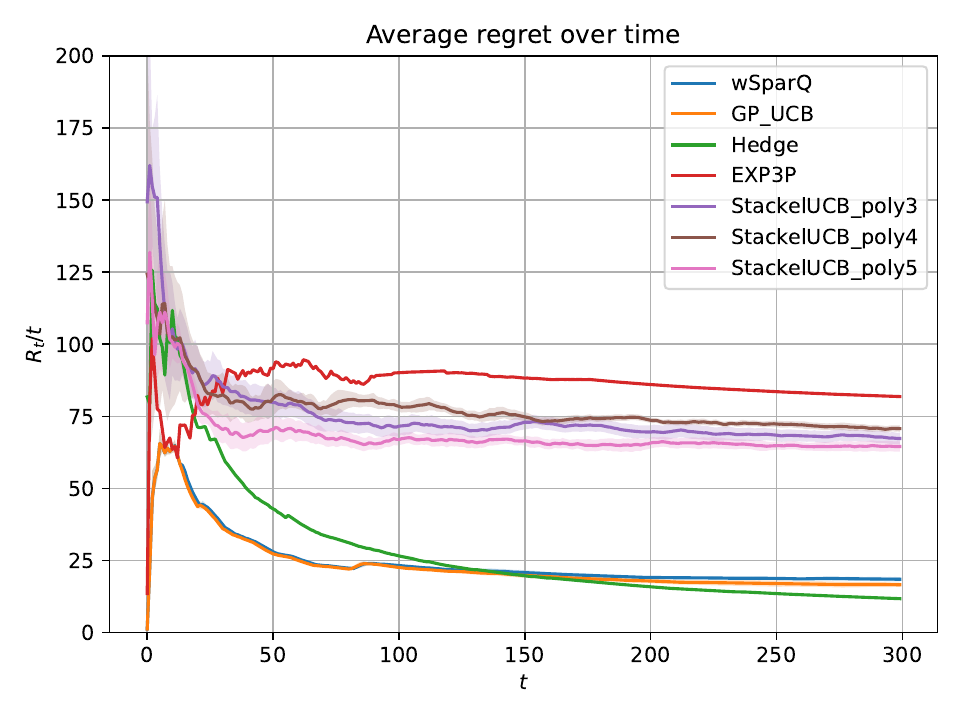}
         \caption{$\beta_t = 0.5,\; d=2$}
         \label{Fig:NoDriftFixedBetad2}
     \end{subfigure}
     \hfill
     \begin{subfigure}[b]{0.32\textwidth}
         \centering
         \includegraphics[width=\textwidth]{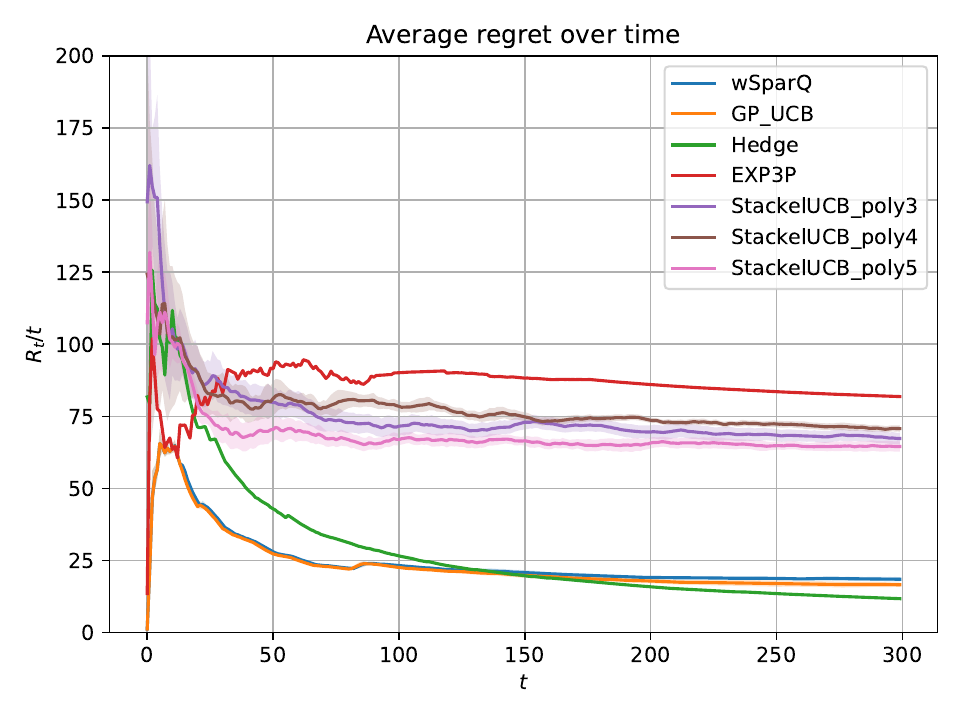}
         \caption{$\beta_t = \log(t),\; d=2$}
         \label{Fig:NoDriftLogBetad2}
     \end{subfigure}
     \caption{Average dynamic cumulative regret for different choices of~$\beta_t$ and~$d$. Figure~\ref{Fig:NoDriftFixedBetad2} shows that \textbf{W-SparQ-BL} achieves similar regret with far fewer additional queries than predicted by theory. Figure~\ref{Fig:NoDriftLogBetad2} illustrates robustness of \textbf{GP-UCB} and \textbf{W-SparQ-BL} with respect to exploration rates. We manually set $\alpha = 1/7$. We compare our method with the algorithm of Sessa et al., \textbf{StackelUCB}, using different polynomial kernel degrees, as well as with the well-known multiplicative weights methods \textbf{EXP3}~\citep{auer2002nonstochastic} and \textbf{HEDGE}~\citep{freund1997decision}.}
     \label{Fig:BiLevelNoDrift}
\end{figure}

While we observe on the plots that \textbf{StackelUCB} methods can actually track the optimum with precision in this setting, expanding the scope of application of their method, we can also see that \textbf{GP-UCB} achieves particularly good performances. Combined with the observation that reduced exploration leads to lower regret, this suggests that variations induced by changes in opponent types~$\theta_t$ can be exploited to learn a near-optimal stationary action. This observation is consistent with the choice of regret notion in~\cite{sessa2020learning}, which compares performance to the best fixed action in hindsight.

\subsection{Time-varying bilevel optimization}\label{App:ExpBL}

The lower-level mappings are constructed to illustrate three regimes. In the stationary setting, the response is smooth and time-invariant:
\[
g(x,t) = \big(\sin(2\pi x), \cos(2\pi x)\big),
\]
which lies naturally in the associated RKHS. 

In the moderate variation setting, we introduce smooth temporal drifts through phase shifts:
\[
g(x,t) = \big(\sin(2\pi x + 0.5t), \cos(2\pi x - 0.03t)\big),
\]
which preserves spatial regularity while inducing a controlled temporal evolution consistent with the noise model in~\eqref{Eq:TVModel} with~$\alpha=1$. 

Finally, fast variation setting combines spatial non-smoothness with stronger temporal drift:
\[
g(x,t) =
\begin{cases}
2x + 0.2\sqrt{t},\; -2x + 0.1t, & x < 0.5, \\
-2x + 2 + 0.2\sqrt{t},\; 2x - 2 + 0.1t, & x \geq 0.5,
\end{cases}
\]
which challenges the RKHS assumption and induces rapidly increasing effective noise.

The reward function is defined as
\[
f(x,y) = - (x - 0.5)^2 - \tfrac{1}{2}\|y\|^2,
\]
which is Lipschitz continuous and allows to directly translate response estimation errors into regret.

\section{Proofs of the theoretical results}\label{App:proofs}

\begin{proof}[Proof of Proposition~\ref{Prop:responseReg}]
Let~$t_1<t_2$,~$x\in\mathcal{X}$, and~$\delta>0$. By Assumption~\ref{As:opponentVar}, 
\begin{equation}\label{Eq:ProbaTheta} 
\mathbb{P}\!\left(|\theta_{t_1}-\theta_{t_2}|>\frac{\delta}{L_g}\right) \leq 2\exp\!\left(-\frac{\delta^2}{2\sigma^2|t_2-t_1|^{\alpha}}\right). 
\end{equation}
By~$L_g$-Lipschitz continuity of~$g$,
\begin{align*} 
|g_{t_2}(x)-g_{t_1}(x)| =|g(x,\theta_{t_2})-g(x,\theta_{t_1})| \leq L_g|\theta_{t_2}-\theta_{t_1}|. 
\end{align*} 
Combining the two inequalities yields the result. 
\end{proof}

\begin{proof}[Proof of Theorem~\ref{Thm:BiLevelRegret}] 
Let~$0< \delta <1$. We start by showing that the bilevel optimization setting can be cast into the time-varying GP framework studied previously. Assumption~\ref{As:ResponseRegularity} and Proposition~\ref{Prop:responseReg} allow to work in the same framework as~\cite{mauduit2025no}.

As a consequence, observations collected at earlier times provide noisy information about future functions, with a noise level that increases as a function of the time lag. We therefore leverage the following time-varying noise model: 
\begin{equation}\label{Eq:TVModel} 
\forall t_1 \leq t_2, \quad y_{t_1} = g_{t_2}(x_{t_1}) + \epsilon_{t_1, t_2}, 
\end{equation} 
where~$\epsilon_{t_1,t_2}$ is a zero-mean sub-Gaussian random variable with variance~$\sigma^2\left( 1 + (t_2-t_1)^\alpha \right)$, and, for fixed~$t_2$, the collection~$\{\epsilon_{t_1, t_2}\}_{t_1<t_2}$ is independent. This corresponds exactly to the model~\eqref{Eq:TVModel} applied to the function sequence~$(g_t)_t$. 

At step~$t$, we construct noise-aware posterior estimates~$\mu_t$ and~$\sigma_t^2$ of~$g_{t+1}$ using a stationary squared exponential kernel: 
\begin{align} 
\mu_t(x) &= k_t(x)^\top (K_t + \Sigma_t)^{-1} Y_t, \\ 
\sigma_t^2(x) &= k(x,x) - k_t(x)^\top (K_t + \Sigma_t)^{-1} k_t(x), 
\end{align} 
where~$\Sigma_t = \diag\left((\sigma^2(1+i^\alpha))_{i=1}^{t}\right)$ is the time-dependent noise matrix,~${K_t = [k(x_i,x_j)]_{i,j \leq t}}$ is the kernel matrix, and~$k_t(x) = [k(x_1,x), \ldots,k(x_t,x)]^\top$. 

Let
\[\beta_t = \sqrt{2 \log \left(  \frac{|\Sigma_t + K_t|^{1/2}}{\delta|K_t|^{1/2}} \right)} + B
\]
and define the confidence bounds for all~$x \in \mathcal{X}$: 
\begin{equation*} 
\text{ucb}_{t+1}(x) = \mu_{t}(x)+\beta_{t+1}\sigma_{t}(x),\quad \text{lcb}_{t+1}(x) = \mu_{t}(x)-\beta_{t+1}\sigma_{t}(x). 
\end{equation*} 
Lemma~7 from \cite{kirschner2018informationdirectedsamplingbandits} guarantees that, with probability at least~$1-\delta/2$, 
\begin{align*} 
\forall x \in \mathcal{X}, \quad g_{t+1}(x) \in [\text{lcb}_{t+1}(x), \text{ucb}_{t+1}(x)]. 
\end{align*} 
Let~$T>1$ be a finite horizon. The dynamic cumulative regret is defined as 
\begin{align*} 
R_T = \sum_{t=1}^T \max_{x \in \mathcal{X}} f(x, g_t(x)) - \sum_{t=1}^T f(x_t, g_t(x_t)). 
\end{align*} 
The instantaneous regret at time~$t$ can be written as 
\begin{align} 
r_t &= \max_{x \in \mathcal{X}} f(x, y_t) - f(x_t, y_t)\\ 
& = f(x_t^*, y_t) - f(x_t, y_t). \label{Eq:InstReg1} 
\end{align} 
With probability at least~$1-\delta/2$,~$y_t \in\mathcal{U}_t(x) = [\text{lcb}_{t}(x), \text{ucb}_{t}(x)]$, and therefore 
\begin{equation}\label{Eq:InstReg2} 
f(x_t, y_t) \in \left[\min_{y \in \mathcal{U}_t(x_t)} f(x_t, y), \max_{y \in \mathcal{U}_t(x_t)} f(x_t, y)\right]. 
\end{equation} 
Combining~\eqref{Eq:InstReg1} and~\eqref{Eq:InstReg2}, we obtain with probability at least~$1-\delta/2$: 
\begin{align*} 
r_t &\leq \max_{y \in \mathcal{U}_t(x_t^*)} f(x_t^*, y) - \min_{y \in \mathcal{U}_t(x_t)} f(x_t, y) \\ 
& \leq \max_{y \in \mathcal{U}_t(x_t)} f(x_t, y) - \min_{y \in \mathcal{U}_t(x_t)} f(x_t, y), 
\end{align*} 
where the second inequality follows from the action selection rule 
\begin{align*} 
x_t = \argmax_{x \in \mathcal{X}} \max_{y \in \mathcal{U}_t(x)} f(x,y). 
\end{align*} 
Define 
\[
y^m_t = \argmin_{y \in \mathcal{U}_t(x_t)} f(x_t,y), \qquad y^M_t = \argmax_{y \in \mathcal{U}_t(x_t)} f(x_t,y). 
\]
Then 
\begin{align} 
r_t &\leq f(x_t, y^M_t) - f(x_t,y^m_t)\\ 
& \leq L_r \|(x_t-x_t, y^M_t - y^m_t)\|_1 \label{Eq:LrLipsch}\\ 
& \leq L_r \left( \text{ucb}_t(x_t) - \text{lcb}_t(x_t) \right)\\ 
& = 2 L_r \beta_t \sigma_{t-1}(x_t), \label{Eq:WidthConfBounds} 
\end{align} where~\eqref{Eq:LrLipsch} follows from the $L_f$-Lipschitz continuity of~$f$ with respect to~$\|\cdot\|_1$, and~\eqref{Eq:WidthConfBounds} from the definition of the confidence bounds. Summing over~$t$ gives 
\begin{equation*} 
R_T \leq 2 L_r \beta_T \sum_{t=1}^T \sigma_{t-1}(x_t). 
\end{equation*} 
The posterior variance~$\sigma_{t-1}$ is computed using the sparse windowed mechanism of \textbf{W-SparQ-GP-UCB}. The sum of posterior variance can be bounded with probability~$1-\delta/2$ as in Theorem~3.8~\citep{mauduit2025no}. Hence, with probability~$1-\delta$, 
\begin{align*} 
R_T = \mathcal{O}\left(\sqrt{T^{\tilde{\alpha}+1} d \log^{d+1 }T\left(\log \frac{1}{\delta} + d^2 \log^{d+1}\left(\log T\right)\right)}\right), 
\end{align*} 
which concludes the proof. \end{proof}

\begin{proof}[Proof of Proposition~\ref{prop:matern}]
\textbf{Step 1: Well-definedness.}
Fix $x \in \mathcal{X}$. By (A1), $g(x,\cdot)$ is $\mu$-strongly convex on $\mathcal{Y}$, hence admits a unique minimizer. This defines $\tilde{g}(x)$.

\textbf{Step 2: Regularity via implicit function theorem.}
Define
\[
F(x,y) := \nabla_y g(x,y).
\]
By (A2), $F \in H^s(\mathcal{X} \times \mathcal{Y})^m$. Moreover,
\[
F(x,\tilde{g}(x)) = 0, \quad \nabla_y F(x,\tilde{g}(x)) = \nabla^2_{yy} g(x,\tilde{g}(x)) \succeq \mu I_m.
\]
Hence, the Jacobian is uniformly invertible. The global implicit function theorem implies that $\tilde{g}$ is $\mathcal{C}^s$ on $\mathcal{X}$, with
\[
D_x \tilde{g}(x) = - \left[\nabla^2_{yy} g(x,\tilde{g}(x))\right]^{-1} \nabla^2_{xy} g(x,\tilde{g}(x)).
\]

\textbf{Step 3: Sobolev regularity.}
Since all derivatives of $g$ up to order $s+1$ are square-integrable, the above expression shows that derivatives of $\tilde{g}$ up to order $s$ are square-integrable as well. Hence $\tilde{g} \in H^s(\mathcal{X})^m$.

The identification with the Matérn RKHS concludes the proof.
\end{proof}

\begin{proof}[Proof of Theorem~\ref{Thm:UCBLRegret}]
In this proof, we use the following notation: for a vector~$v \in \R^d$, ve denote by~$v^i$ its $i$-th coordinate, with~$i \leq m$.

For each coordinate $i$, define the event
\[
\mathcal{E}_{t,i} := \{\forall x \in \mathcal{X}, \tilde{g}_i(x) \in \tilde{\mathcal{U}}_t^i(x)\}.
\]
By Lemma 7 of~\cite{kirschner2018informationdirectedsamplingbandits},
\[
\mathbb{P}(\mathcal{E}_{t,i}) \geq 1 - \frac{\delta}{2m}.
\]
By union bound, the event 
\begin{align*}
\mathcal{E}_t &= \{\forall x \in \mathcal{X}, \tg(x) \in \tilde{\mathcal{U}}_t(x)\}\\ &= \bigcap_{i=1}^m \mathcal{E}_{t,i},
\end{align*}
holds with probability at least~$1-\frac{\delta}{2}$.
We consider the instantaneous regret at~$t$ 
\[ r_t = \underset{x \in \mathcal{X}}{\max} f(x,y_t) - f(x_t, y_t), 
\] 
then, with probability at least~$1-\frac{\delta}{2}$, 
\[ r_t(x_t, y_t) \in \left[\underset{y \in \tilde{\mathcal U}_t(x_t)}{\min} f(x_t,y),\underset{y \in \tilde{\mathcal U}_t(x_t)}{\max} f(x_t,y)\right]. 
\] 
Then, if~$x^* = \underset{x \in \mathcal{X}}{\argmax} \, f(x)$, still with probability at least~$1-\frac{\delta}{2}$, 
\begin{align*} 
r_t &\leq \underset{y \in \tilde{\mathcal U}_t(x^*)}{\max} f(x^*,y) - \underset{y \in \tilde{\mathcal U}_t(x_t)}{\min} f(x_t,y)\\ 
&\leq \underset{y \in \tilde{\mathcal U}_t(x_t)}{\max} f(x_t,y) - \underset{y \in \tilde{\mathcal U}_t(x_t)}{\min} f(x_t,y), 
\end{align*} 
where the second inequality comes from the selection rule~\eqref{Eq:BiLevelRule}. If we let 
\[ y_{t,m} = \underset{y \in \tilde{\mathcal U}_t(x_t)}{\argmin} \, f(x_t,y), \qquad y_{t,M} =\underset{y \in \tilde{\mathcal U}_t(x_t)}{\argmax} \, f(x_t,y), 
\] 
by $L_f$-Lipschitzness of~$f$, 
\begin{align*} 
r_t &\leq L_f \left( \|x_t-x_t\|_1 + \|y_{t,M}-y_{t,m}\|_1\right) \\ 
&= L_f \sum_{i=1}^m \left( y_{t,M}^i-y_{t,m}^i \right)\\ 
&\leq L_r \sum_{i=1}^m \left( \text{ucb}_t^i(x_t) - \text{lcb}_t^i(x_t)\right) \\ 
&= L_r \sum_{i=1}^m 2 \beta_t \sigma^i_{t-1}(x_t). 
\end{align*} 
By summing over time horizon~$T$, we obtain 
\begin{align*} 
R_T &\leq 2 \beta_T L_f \sum_{t=1}^T \sum_{i=1}^m \sigma^i_{t-1}(x_t)\\ 
&= \sum_{i=1}^m R_{T,i}, 
\end{align*} 
with~$R_{T,i}$ the cumulative regret of the regular univariate \textbf{GP-UCB} on the $i$-th coordinate. 
Applying Theorem~3 from~\cite{srinivas2012information} and using~$\beta_T = \mathcal{O}\left( \sqrt{\gamma_{Q_T} + \log \frac{1}{\delta}} \right)$, we obtain 
\[
 R_T = \mathcal{O}\left( m \sqrt{T \gamma_T \left( 2 \log \frac{1}{\delta} + \gamma_T\right)}\right). 
\]
\end{proof}

\begin{proof}[Proof of Proposition~\ref{Prop:BiLevelTemporalVar}]
Fix $i$ and $x \in \mathcal{X}$. By the reproducing property,
\[
\tilde{g}_{t+n}^i(x) - \tilde{g}_t^i(x)
= \langle \tilde{g}_{t+n}^i - \tilde{g}_t^i, k(x,\cdot) \rangle.
\]
Hence,
\[
|\tilde{g}_{t+n}^i(x) - \tilde{g}_t^i(x)|
\leq \|\tilde{g}_{t+n}^i - \tilde{g}_t^i\|_{\Hk} \|k(x,\cdot)\|_{\Hk}.
\]
Using $\|k(x,\cdot)\|_{\Hk}^2 = k(x,x) \leq M_k$,
\[
\leq \sqrt{M_k} \|\tilde{g}_{t+n}^i - \tilde{g}_t^i\|_{\Hk}.
\]
Finally, by triangle inequality and uniform boundedness,
\[
\|\tilde{g}_{t+n}^i - \tilde{g}_t^i\|_{\Hk}
\leq \sum_{s=t}^{t+n-1} \|\tilde{g}_{s+1}^i - \tilde{g}_s^i\|_{\Hk}
\leq 2 n B.
\]
\end{proof}

\begin{proof}[Proof of Theorem~\ref{Thm:WSPARQBLRegret}]
The proof follows the same ideas as those from Theorems~\ref{Thm:BiLevelRegret} and~\ref{Thm:UCBLRegret}. 

By performing \textbf{W-SparQ-GP-UCB} element wise, with appropriate noise variance proxy control through parameter $\tilde\alpha$ and collecting a sufficient number of additional queries at the beginning of each window, we bound the element-wise dynamic cumulative regrets 
\[
R_{T,i} \leq 2 L_f \beta_T \sum_{t=1}^T \mathbf{\sigma}^i_{t-1}(x_t),
\]
where the posterior variance $\mathbf{\sigma}^i_{t-1}(.)$ is computed using the sparse windowed mechanism of \textbf{W-SparQ-GP-UCB} adapted for the Matérn-$\nu$ kernel, with the noise proxy bounded by $\tilde{\sigma}_t^2 = \mathcal{O}(t^{\tilde{\alpha}})$, yielding
\[
R_{T,i} = \mathcal{O}\left( \sqrt{T^{\tilde{\alpha}+1} \gamma_T \left(\gamma_{Q_T} + \log \frac{1}{\delta}\right)}\right),
\]
where $Q_T = \mathcal{O}(\log^d(T))$ for the SE kernel and $Q_T = \mathcal{O}\left(T^{\frac{2d}{2\nu-d}}\right)$
Now, following the same steps as in proof of Theorem~\ref{Thm:UCBLRegret}, we can bound the regret of \textbf{W-SparQ-BL} as
\begin{align*}
R_T &\leq \sum_{i=1}^m R_{T,i}\\
&= \mathcal{O}\left(m \sqrt{T^{\tilde{\alpha}+1} \gamma_T \left(\gamma_{Q_T} + \log \frac{1}{\delta}\right)}\right).
\end{align*}
\end{proof}

\end{document}

%% file: macros.tex
\usepackage{amsmath}    
\usepackage{amssymb}    
\usepackage{amsfonts}
\usepackage{amsthm}     
\usepackage{mathtools}  
\usepackage{dsfont}     
\DeclareMathAlphabet{\mathcal}{OMS}{cmsy}{m}{n} 

\theoremstyle{plain}
\newtheorem{theorem}{Theorem}[section]
\newtheorem{proposition}[theorem]{Proposition}

\theoremstyle{definition}

\newtheorem{assumption}[theorem]{Assumption}

\theoremstyle{remark}

\usepackage{algorithm}
\usepackage{algpseudocode}

\usepackage{dirtree}
\usepackage{graphicx}   
\usepackage{float}      
\usepackage{subcaption} 
\usepackage{epstopdf}   
\usepackage{caption}

\usepackage{booktabs}   
\usepackage{multirow}
\usepackage{adjustbox}
\usepackage{threeparttable}
\usepackage{array}
\usepackage{arydshln}   

\usepackage{enumitem}
\setlist[enumerate]{leftmargin=.5in}
\setlist[itemize]{leftmargin=.5in}

\usepackage{xcolor}
\definecolor{myblue}{HTML}{1F77B4}
\definecolor{myred}{HTML}{D62728}
\definecolor{mygreen}{HTML}{2CA02C}

\usepackage{tikz}
\usetikzlibrary{arrows.meta,arrows,shapes,snakes,automata,backgrounds,petri,positioning,decorations.markings}
\usepackage{quantikz} 

\usepackage{graphicx}
\usepackage{ragged2e}
\usepackage{comment}
\usepackage{cases}
\usepackage{mathabx} 
\usepackage{dsfont}  
\usepackage{pdfpages}

\providecommand{\Dset}{\mathcal{D}}

\providecommand{\Xs}{{\color{myred}X}^\mathrm{s}}
\providecommand{\Ys}{{\color{myred}Y}^\mathrm{s}}
\providecommand{\Thetas}{{\color{myred}\Theta}^\mathrm{s}}

\providecommand{\Xr}{{\color{mygreen}X}^\mathrm{r}}
\providecommand{\Yr}{{\color{mygreen}Y}^\mathrm{r}}
\providecommand{\Thetar}{{\color{mygreen}\Theta}^\mathrm{r}}

\newcommand{\R}{\mathbb{R}}

\newcommand{\tg}{\tilde{g}}
\newcommand{\Hk}{\mathcal{H}}

\DeclareMathOperator*{\argmax}{arg\,max}
\DeclareMathOperator*{\argmin}{arg\,min}
\DeclareMathOperator{\diag}{diag}

\usepackage{cleveref} 
\crefname{assumption}{Assumption}{Assumptions}
\creflabelformat{equation}{#2#1#3}

